\documentclass[12pt]{amsart}
\usepackage{amssymb,amsmath,amsfonts,latexsym}
\usepackage{bm, enumerate}
 \usepackage{xcolor}

\newcommand{\newsection}[1]{\setcounter{equation}{0} \section{#1}}
\newcommand{\bea}{\begin{eqnarray}}
	\newcommand{\eea}{\end{eqnarray}}

\newcommand{\clb}{\mathcal{B}}

\newcommand{\clh}{\mathcal{H}}

\newcommand{\D}{\mathbb{D}}
\newcommand{\N}{\mathbb{N}}
\newcommand{\C}{\mathbb{C}}

\newcommand{\raro}{\rightarrow}

\def\textmatrix#1&#2\\#3&#4\\{\bigl({#1 \atop #3}\ {#2 \atop #4}\bigr)}
\def\dispmatrix#1&#2\\#3&#4\\{\left({#1 \atop #3}\ {#2 \atop #4}\right)}
\newcommand{\be}{\begin{equation}}
	\newcommand{\ee}{\end{equation}}
\newcommand{\ben}{\begin{eqnarray*}}
	\newcommand{\een}{\end{eqnarray*}}

\newcommand{\NI}{\noindent}

\newcommand{\bi}{\begin{itemize}}
	\newcommand{\ei}{\end{itemize}}
\newcommand{\diag}{\mbox{diag}}

\newcommand\la{{\langle }}
\newcommand\ra{{\rangle}}

\newtheorem{Theorem}{\sc Theorem}[section]
\newtheorem{Lemma}[Theorem]{\sc Lemma}
\newtheorem{Proposition}[Theorem]{\sc Proposition}
\newtheorem{Corollary}[Theorem]{\sc Corollary}
\newtheorem{Example}[Theorem]{\sc Example}
\newtheorem{Remark}[Theorem]{\sc Remark}

\newtheorem{Note}[Theorem]{\sc Note}
\newtheorem{Question}{\sc Question}
\newtheorem{ass}[Theorem]{\sc Assumption}
\newtheorem{Definition}[Theorem]{\sc Definition}
\newcommand{\bt}{\begin{Theorem}}
\def\beginlem{\begin{Lemma}}
\def\beginprop{\begin{Proposition}}
\def\begincor{\begin{Corollary}}
\def\begindef{\begin{Definition}}
\def\beginexamp{\begin{Example}}
\def\beginrem{\begin{Remark}}
\def\beginq{\begin{Question}}
\def\beginass{\begin{ass}}
\def\beginnote{\begin{Note}}
\newcommand{\et}{\end{Theorem}}
\def\endlem{\end{Lemma}}
\def\endprop{\end{Proposition}}
\def\endcor{\end{Corollary}}
\def\enddef{\end{Definition}}
\def\endexamp{\end{Example}}
\def\endrem{\end{Remark}}
\def\endq{\end{Question}}
\def\endass{\end{ass}}
\def\endnote{\end{Note}}

\begin{document}

\title{Tridiagonal shifts as compact + isometry}

\author[Das]{Susmita Das}
\address{Indian Statistical Institute, Statistics and Mathematics Unit, 8th Mile, Mysore Road, Bangalore, 560059,
India}
\email{susmita.das.puremath@gmail.com}

\author[Sarkar]{Jaydeb Sarkar}
\address{Indian Statistical Institute, Statistics and Mathematics Unit, 8th Mile, Mysore Road, Bangalore, 560059,
India}
\email{jay@isibang.ac.in, jaydeb@gmail.com}


\subjclass{Primary: 46E22, 47B38, 47A55; Secondary: 47B07, 30H10, 47B37}
\keywords{Tridiagonal kernels, perturbations, compact operators, isometries, shifts}
	
\begin{abstract}
Let $\{a_n\}_{n\geq 0}$ and $\{b_n\}_{n\geq 0}$ be sequences of scalars. Suppose $a_n \neq 0$ for all $n \geq 0$. We consider the  tridiagonal kernel (also known as band kernel with bandwidth one) as
\[
k(z, w) = \sum_{n=0}^\infty ((a_n + b_n z)z^n) \overline{(({a}_n + {b}_n {w}) {w}^n)} \qquad (z, w \in \mathbb{D}),
\]
where $\mathbb{D} = \{z \in \mathbb{C}: |z| < 1\}$. Denote by $M_z$ the multiplication operator on the reproducing kernel Hilbert space corresponding to the kernel $k$. Assume that $M_z$ is left-invertible. We prove that $M_z =$ compact $+$ isometry if and only if $|\frac{b_n}{a_n}-\frac{b_{n+1}}{a_{n+1}}|\rightarrow 0$ and $|\frac{a_n}{a_{n+1}}| \rightarrow 1$.
\end{abstract}
	
\maketitle

\newsection{Introduction }\label{sec: intro}

A bounded linear operator $T$ on a Hilbert space $\clh$ ($T \in \clb(\clh)$ in short) is called \textit{semi-Fredholm} if the range space $\text{ran}T$ is closed and at least one of the spaces $\ker T$ and $\ker T^*$ is of finite dimension. If $T$ is semi-Fredholm then
\[
ind (T) = \text{dim} \ker T - \text{dim} \ker T^*,
\]
is called the \textit{index} of $T$. We shall always assume that our Hilbert spaces are separable and over $\mathbb{C}$. The starting point of our present note is the following classification of compact perturbations of isometries \cite[page 191]{Fillmore}:

\begin{Theorem}[Fillmore, Stampfli, and Williams]\label{thm: classic}
Let $T \in \clb(\clh)$. Then $T = $ compact $+$ isometry if and only if $I - T^*T$ is compact and $T$ is semi-Fredholm with $ind (T) \leq 0$.
\end{Theorem}

In this note, we are interested in a quantitative version of the above theorem. For instance, consider a bounded sequence of non-zero scalars $\{w_n\}_{n \geq 0}$ and an infinite-dimensional Hilbert space $\clh$ with an orthonormal basis $\{e_n\}_{n\geq 0}$. Then the \textit{weighted shift} $S_w$ defined by
\[
S_w (e_n) = w_n e_{n+1} \qquad (n \geq 0),
\]
is in $\clb(\clh)$ with $\|S_w\| = \sup_{n} |w_n|$. Since $\ker S_w = \{0\}$ and $\ker S_w^* = \{e_0\}$, it follows that $S_w$ is semi-Fredholm and $\text{ind}(S_w) = -1$. Moreover, using the fact that $S_w^* e_0 = 0$ and $S_w^* e_n = \bar{w}_{n-1} e_{n-1}$, $n \geq 1$, it follows that
\[
I- S_w^*S_w = \diag (1- |w_0|^2, 1- |w_1|^2, \ldots).
\]
Theorem \ref{thm: classic} then readily implies that
\begin{equation}\label{eqn: S_w = C+I}
\lim_{n \raro \infty} |w_n| = 1 \; \text{ if and only if } \; S_w = \text{compact} + \text{isometry}.
\end{equation}
We note that in this case the weight sequence is bounded away from zero and hence $S_w$ is necessarily left-invertible.

Also note that $S_w$ is a concrete example of a left-invertible shift on an analytic Hilbert space. To be more precise, let $\D$ denote the open unit disc in $\mathbb{C}$. A function $k : \D \times \D \raro \mathbb{C}$ is called an \textit{analytic kernel} if $k$ analytic in the first variable and $k$ is positive definite, that is
\[
\sum_{i,j=1}^n c_j \bar{c}_ik(z_i, z_j) \geq 0,
\]
for all $\{z_i\}_{i=1}^n \subseteq \D$, $\{c_i\}_{i=1}^n \subseteq \mathbb{C}$ and $n \geq 1$. In this case, there exists an \textit{analytic Hilbert space} $\clh_k$ of analytic functions on $\D$ such that $\{k(\cdot, w) : w \in \D\}$ is a total set in $\clh_k$. The \textit{shift} on $\clh_k$ is defined by $M_z f = z f$, $f \in \clh_k$. We will always assume in the sequel that $M_z$ is bounded. Since $k$ is analytic in the first and (automatically) co-analytic in the second variable, it follows that
\[
k(z,w) = \sum_{m,n \geq 0} c_{mn} z^m \bar{w}^n \qquad (z, w \in \D).
\]
When $c_{mn} = 0$ for all $|m-n| \geq 2$ ($|m-n| \geq 1$), we say that $\clh_k$ is a \textit{tridiagonal space} (\textit{diagonal space}) and $k$ is a \textit{tridiagonal kernel} or \textit{band kernel with bandwidth $1$} (\textit{diagonal kernel}).

A standard computation now reveals that $S_w$, under some appropriate assumption on the weight sequence $\{w_n\}_{n\geq 0}$ \cite[proposition 7]{AS}, is unitarily equivalent to $M_z$ on a diagonal space. Therefore \eqref{eqn: S_w = C+I} yields a quantitative classification of shifts on diagonal spaces that are compact perturbations of isometries. This motivates the following natural question:

\begin{Question}\label{Quest}
Is it possible to find a quantitative classification of left-invertible shifts on analytic Hilbert spaces that are compact perturbations of isometries?
\end{Question}

The main purpose of this note is to provide an answer to the above question for the case of $M_z$ on (tractable) tridiagonal spaces. Throughout the paper, we fix sequences of scalars $\{a_n\}_{n \geq 0}$ and $\{b_n\}_{n \geq 0}$ with the assumption that $a_n \neq 0$, $n \geq 0$. We set
\[
f_n(z) = (a_n + b_n z) z^n \quad \quad (n \geq 0),
\]
and consider the Hilbert space $\clh_k$ with $\{f_n\}_{n\geq 0}$ as an orthonormal basis. Then $\clh_k$ is a tridiagonal space corresponding to the tridiagonal kernel
\[
k(z, w) = \sum_{n=0}^\infty f_n(z) \overline{f_n(w)} \qquad (z, w \in \mathbb{D}).
\]
We always assume that $\{|\frac{a_n}{a_{n+1}}|\}_{n\geq 0}$ is bounded away from zero and
\[
\sup_{n\geq 0} | \frac{a_n}{a_{n+1}}| < \infty \text{ and } \limsup_{n\geq 0} | \frac{b_n}{a_{n+1}}| < 1.
\]
The latter two assumptions ensure that $M_z$ on $\clh_k$ is bounded \cite[Theorem 5]{Adam 2001}, whereas the first assumption implies that $M_z$ is left-invertible \cite[Theorem 3.5]{Das-Sarkar}. In this case we also call $M_z$ a \textit{tridiagonal shift}.

The notion of tridiagonal shifts was introduced by Adams and McGuire \cite{Adam 2001}. A part of their motivation came from factorizations of positive operators on analytic Hilbert spaces \cite{AMP} (also see \cite{PW}). Evidently, if $b_n = 0$, then $k$ is a diagonal kernel and $M_z$ is a weighted shift on $\clh_k$. Therefore, in view of shifts on analytic Hilbert spaces, tridiagonal shifts are the ``next best'' concrete examples of shifts after weighted shifts. The following is the answer to Question \ref{Quest} for tridiagonal shifts:

\begin{Theorem}[Main result]\label{trid M_z = S+K}
Let $M_z$ be the tridiagonal shift on $\clh_k$. Then $M_z = \text{compact} + \text{isometry}$ if and only if $|\frac{a_n}{a_{n+1}}| \raro 1$ and $|\frac{b_n}{a_n}-\frac{b_{n+1}}{a_{n+1}}| \raro 0$.
\end{Theorem}

In Section \ref{sec: proof}, we present the proof of the above theorem. In Section \ref{sec: Prep} we prove a key proposition that says that if $T \in \clb(\clh)$ is a left-invertible operator and if $T$ is of finite index, then $T =$ compact $+$ isometry if and only if $L_T - T^*$ is compact, where $L_T = (T^*T)^{-1} T^*$. Section \ref{sec: concl} concludes the paper with some general remarks and additional observations.

\newsection{Preparatory results}\label{sec: Prep}

The aim of this section is to prove a key result of this paper. We begin with some elementary properties of left-invertible operators. See \cite{Shim} for more on this theme. Let $T \in \clb(\clh)$ be a left-invertible operator. We use the fact that $T^*T$ is invertible to see that
\[
L_T = (T^*T)^{-1} T^*,
\]
is a left inverse of $T$. Note that $(T L_T)^2 = T L_T = (T L_T)^*)$, that is, $T L_T$ is an orthogonal projection. Moreover, if $T^*f = 0$ for some $f \in \clh$, then $(I - T L_T)f = f$. On the other hand, if $(I - T L_T)f = f$ for some $f \in \clh$, then $T L_Tf = 0$ and hence $T^*T L_T f = 0$, which implies that $T^*f = 0$. Therefore, $I - T L_T$ is the orthogonal projection onto $\ker T^*$, that is
\[
I - T L_T = P_{\ker T^*}.
\]

Part of the following is a particular case of \cite[Theorem 6.2]{Fillmore}. However, part (3) appears to be new, which will be also a key to the proof of the main theorem of this paper. For the sake of completeness, we present the argument with all details.

\begin{Proposition}\label{thm: $T = S+K$}
Let $T \in \clb(\clh)$ be left-invertible and of finite index. The following
statements are equivalent:
\begin{enumerate}
\item $T = $  compact $+$ isometry.
\item $I - T^*T$ is compact.
\item $L_T-T^*$ is compact.
\item $I - TT^*$ is compact.
\end{enumerate}	
\end{Proposition}	

\begin{proof}
Throughout the following, we will designate compact operators by letters such as $K, K_1, K_2$, etc.

\noindent $(1) \Rightarrow (2)$: Suppose $T = S+K$ for some isometry $S$ on $\clh$. Then
\[
T^*T =(S+K)^*(S+K) = S^*S + K_1 = I + K_1,
\]
implies that $I - T^*T$ is compact.

\noindent $(2) \Rightarrow (3)$: Since $I - T L_T = P_{\ker T^*}$ and $\text{dim} \ker T^* < \infty$, we have $T L_T = I + K_1$. Now if $I - T^*T = K_2$, then $L_T - T^*T L_T = K_3$, and hence
\[
K_3 = L_T - T^*T L_T =  L_T - T^*(I+K_1) = L_T - T^* + K_4.
\]
This gives us $L_T - T^* = K$.

\noindent To prove $(3) \Rightarrow (4)$, assume that $L_T - T^* = K$. Then $T L_T - T T^* = K_1$. Again, since $I - T L_T = P_{\ker T^*}$ and $\text{dim} \ker T^* < \infty$, we have
\[
I - T T^* = (I - T L_T) + (T L_T - T T^*) = P_{\ker T^*}+ K_1 = K_2.
\]

\noindent $(4) \Rightarrow (2)$: Let $K = I - TT^*$. Then $T^* K = T^* - T^* TT^* = (I - T^*T)T^*$ implies that $T(I - T^* T) = K_1$, and hence
\[
I - T^* T = L_T T(I - T^*T) = L_T K_1 = K_2.
\]
$(2) \Rightarrow (1)$ Suppose $I - T^*T = K$. Since $|T|$ is positive, we see that $(I+|T|)$ is invertible. Then $K = (I+|T|)(I - |T|)$ implies that $|T| = I + K_1$. Let $T = U |T|$ be the polar decomposition of $T$. Taking the injectivity property of $T$ in account, we find that $U$ is an isometry, which implies
\[
T = U|T| = U (I + K_1) = U + K_2,
\]
and completes the proof of the proposition.
\end{proof}	

Unlike the proof of \cite{Fillmore}, the above proof avoids employing the Calkin algebra method. Of course, as pointed out earlier, the result of \cite{Fillmore} (modulo part (3)) holds without the left-invertibility assumption.

Now we turn to the tridiagonal shift $M_z$ on $\clh_k$, where
\[
k(z, w) = \sum_{n=0}^{\infty} f_n(z) \overline{f_n(w)} \qquad (z, w \in \D),
\]
and $f_n(z) = (a_n + b_nz)z^n$, $a_n, b_n \in \mathbb{C}$, $n \geq 0$. Recall that $a_n \neq 0$ for all $n \geq 0$. Moreover, by assumption, $\{| \frac{a_n}{a_{n+1}}|\}_{n\geq 0}$ is bounded away from zero and $\sup_{n\geq 0} | \frac{a_n}{a_{n+1}}| < \infty$ and $\limsup_{n\geq 0}| \frac{b_n}{a_{n+1}}| < 1$, which ensures that $M_z$ is bounded and left-invertible on $\clh_k$. It will be convenient to work with the matrix representation of $M_z$ with respect to the orthonormal basis $\{f_n\}_{n \geq 0}$. A standard computation reveals that \cite[Section 3]{Adam 2001}
\[
z^n = \frac{1}{a_n} \sum_{m=0}^\infty (-1)^m \Big(\frac{\prod_{j=0}^{m-1} b_{n+j}}{\prod_{j=0}^{m-1} a_{n+j+1}}\Big) f_{n+m} \quad \quad (n \geq 0),
\]
where $\prod_{j=0}^{-1} x_{n+j} := 1$. A new round of computation then gives
\[
M_z f_n  = \Big(\frac{a_n}{a_{n+1}}\Big) f_{n+1} + c_n \sum_{m=0}^\infty (-1)^m \Big(\frac{\prod_{j=0}^{m-1} b_{n+2+j}}{\prod_{j=0}^{m-1} a_{n+3+j}}\Big) f_{n+2+m} \quad \quad (n\geq 0),
\]
where
\begin{equation}\label{eq: c_b a}
c_n = \frac{a_n}{a_{n+2}} \Big(\frac{b_n}{a_{n}} - \frac{b_{n+1}}{a_{n+1}}\Big) \quad \quad (n \geq 0).
\end{equation}
Therefore
\begin{equation}\label{eqn: Mz}
[M_z] = \begin{bmatrix}
0& 0 & 0 & 0 & \dots
\\
\frac{a_0}{a_1} & 0 & 0 & 0 & \ddots
\\
{c_0} & \frac{a_1}{a_2} & 0 & 0 & \ddots
\\
\frac{-c_0 b_2}{a_3} & c_1 & \frac{a_2}{a_3} & 0 & \ddots
\\
\frac{c_0b_2b_3}{a_3a_4} &\frac{-c_1b_3}{a_4} & c_2 & \frac{a_3}{a_4} & \ddots
\\
\frac{-c_0b_2b_3b_4}{a_3a_4a_5} &\frac{c_1b_3b_4}{a_4a_5} & \frac{-c_2b_4}{a_5} & c_3 & \ddots
\\
\vdots & \vdots & \vdots&\ddots &\ddots
\end{bmatrix},
\end{equation}
with respect to the orthonormal basis $\{f_n\}_{n \geq 0}$ \cite[Page 729]{Adam 2001}.

\newsection{Proof of the main theorem}\label{sec: proof}

Now we are ready to prove the main result of this paper. Throughout the proof, we will frequently use matrix representations of bounded linear operators on the tridiagonal space (as well as subspaces of) $\clh_k$ as in \eqref{eqn: Mz}.

\begin{proof}[Proof of Theorem \ref{trid M_z = S+K}]
Since $\ker M_z^* = \mathbb{C}f_0$, we see that $\text{ind} (M_z) = -1$. Using the left-invertibility of $M_z$ applied to Proposition \ref{thm: $T = S+K$}, we see that $M_z = $ isometry + compact if and only if $L_{M_z} - M_z^*$ is compact. By \eqref{eqn: Mz}, the matrix representation of $M_z^*$ is given by
\begin{equation}\label{eqn: Mz^* matrix}
[M^*_z] = \begin{bmatrix}
0& \frac{\bar{a}_0}{\bar{a}_1} & \bar{c}_0 & \frac{-\bar{c}_0 \bar{b}_2}{\bar{a}_3} & \frac{\bar{c}_0 \bar{b}_2 \bar{b}_3}{\bar{a}_3 \bar{a}_4}  & \dots
\\
0& 0& \frac{\bar{a}_1}{\bar{a}_2} & \bar{c}_1 & \frac{-\bar{c}_1 \bar{b}_3}{\bar{a}_4} & \ddots
\\
0 & 0 & 0 & \frac{\bar{a}_2}{\bar{a}_3} & \bar{c}_2 & \ddots
\\
0& 0 & 0 & 0 & \frac{\bar{a}_3}{\bar{a}_4} & \ddots
\\
\vdots & \vdots & \vdots&\vdots &\ddots &\ddots
\end{bmatrix}.
\end{equation}
Recall that $L_{M_z} = (M_z^*M_z)^{-1}M_z^*$ is a left-inverse of $M_z$. It follows that the matrix representation of $L_{M_z}$ with respect to the orthonormal basis $\{f_n\}_{n \geq 0}$ \cite[Theorem 3.5]{Das-Sarkar} is given by
\[
[L_{M_z}] = \begin{bmatrix}
0 & \frac{a_1}{a_0} & 0 & 0  & 0 & \dots
\\
0 & {d_1} & \frac{a_2}{a_1} & 0  & 0 & \ddots
\\
0 & \frac{-d_1b_1}{a_2} & d_2 & \frac{a_3}{a_2} & 0 & \ddots
\\
0 & \frac{d_1b_1b_2}{a_2a_3} &\frac{-d_2b_2}{a_3} & d_3 &\frac{a_4}{a_3} & \ddots
\\
0 & \frac{-d_1b_1b_2b_3}{a_2a_3a_4} &\frac{d_2b_2b_3}{a_3a_4} & \frac{-d_3b_3}{a_4} & d_4 & \ddots
\\
\vdots & \vdots & \vdots&\vdots &\ddots &\ddots
\end{bmatrix},
\]
where $d_n = \frac{b_n}{a_n} - \frac{b_{n-1}}{a_{n-1}}$ for all $n \geq 1$. Therefore, we have the following matrix representation of $L_{M_z} - M_z^*$:
\[
[L_{M_z} -M_z^*] = \begin{bmatrix}
0 & (\frac{a_1}{a_0}-\frac{\bar{a}_0}{\bar{a}_1}) & -\bar{c}_0 & \frac{\bar{c}_0{\bar{b}_2}}{\bar{a}_3} & -\frac{\bar{c}_0 \bar{b}_2 \bar{b}_3}{\bar{a}_3\bar{a}_4} & \dots
\\
0 & {d_1} & (\frac{a_2}{a_1}-\frac{\bar{a}_1}{\bar{a}_2}) & -\bar{c}_1 & \frac{\bar{c}_1 {\bar{b}_3}}{\bar{a}_4}  & \ddots
\\
0 & \frac{-d_1b_1}{a_2} & d_2 & (\frac{a_3}{a_2}-\frac{\bar{a}_2}{\bar{a}_3}) & -\bar{c}_2 & \ddots
\\
0 & \frac{d_1b_1b_2}{a_2a_3} &\frac{-d_2b_2}{a_3} & d_3 & (\frac{a_4}{a_3}-\frac{\bar{a}_3}{\bar{a}_4}) & \ddots
\\
0 & \frac{-d_1b_1b_2b_3}{a_2a_3a_4} &\frac{d_2b_2b_3}{a_3a_4} & \frac{-d_3b_3}{a_4} & d_4 & \ddots
\\
\vdots & \vdots & \vdots&\vdots &\ddots &\ddots
\end{bmatrix}.
\]
Finally, by \eqref{eq: c_b a}, we see that $c_n = \frac{a_n}{a_{n+2}} (\frac{b_n}{a_{n}} - \frac{b_{n+1}}{a_{n+1}})$ for all $n \geq 0$, and hence
\begin{equation}\label{eqn: d =c}
d_{n+1}= - \frac{a_{n+2}}{a_n}c_n \qquad (n \geq 0).
\end{equation}
Now suppose that $L_{M_z} -M_z^*$ is compact. Since $\{f_n\}_{n \geq 0}$ is an orthonormal basis of $\clh_k$, a well-known property of compact operators on Hilbert spaces implies that
\[
\|(L_{M_z} - M_z^*) f_n\| \raro 0 \text{ as } n \raro \infty.
\]
For each $n \geq 1$, use the matrix representation of $L_{M_z} -M_z^*$ to see that
\[
\begin{split}
\|(L_{M_z} -M_z^*)f_{n+2}\|^2 = & \Big|\frac{c_0 b_2 b_3\cdots b_{n+1}}{a_3a_4a_5 \cdots a_{n+2}}\Big|^2 + \cdots + \Big|\frac{c_{n-1} b_{n+1}}{a_{n+2}}\Big|^2 + |c_n|^2
\\
& + \Big|\frac{a_{n+2}}{a_{n+1}}-\frac{\bar{a}_{n+1}}{\bar{a}_{n+2}}\Big|^2+ |d_{n+2}|^2 + \cdots.
\end{split}
\]
In particular
\[
\|(L_{M_z} -M_z^*)f_{n+2}\|^2 \geq |c_n|^2 + \Big|\frac{a_{n+2}}{a_{n+1}}-\frac{\bar{a}_{n+1}}{\bar{a}_{n+2}}\Big|^2 \qquad (n \geq 1),
\]
and hence, $|c_n| \raro 0$ and $|\frac{a_{n+2}}{a_{n+1}}-\frac{\bar{a}_{n+1}}{\bar{a}_{n+2}}| \raro 0$ as $n \raro \infty$. Then we have
\[
\Big| \Big|\frac{a_{n+1}}{a_{n+2}}\Big|^2 - 1 \Big| = \Big|\frac{a_{n+1}}{a_{n+2}}\Big| \Big|\frac{\bar{a}_{n+1}}{\bar{a}_{n+2}} - \frac{a_{n+2}}{a_{n+1}}\Big| \leq \Big|\frac{\bar{a}_{n+1}}{\bar{a}_{n+2}} - \frac{a_{n+2}}{a_{n+1}}\Big| \Big(\sup_m\Big|{\frac{a_m}{a_{m+1}}}\Big|\Big),
\]
and hence $|\frac{a_n}{a_{n+1}}| \raro 1$. Finally, $|c_n| \raro 0$ (see the definition of $c_n$ in \eqref{eq: c_b a}) and the fact that $\{\frac{a_{n}}{a_{n+2}}\}_{n \geq 0}$ is bounded imply that $|\frac{b_n}{a_n} - \frac{b_{n+1}}{a_{n+1}}| \raro 0$.

For the converse direction, we assume that $|\frac{a_n}{a_{n+1}}| \raro 1$ and $|\frac{b_n}{a_n} - \frac{b_{n+1}}{a_{n+1}}| \raro 0$. Taken together, these conditions mean that $|c_n| \raro 0$ (see \eqref{eq: c_b a}). We claim that $L_{M_z} - M_z^*$ is compact. To prove this, we first let $(\mathbb{C} f_0)^\perp = \clh$. Then, with respect to
\[
\clh_k = \C f_0 \oplus \clh,
\]
the operator $L_{M_z}-M_z^*$ can be represented as
\[
L_{M_z}-M_z^* = \begin{bmatrix}
0 & A
\\
0 & B
\end{bmatrix},
\]
where $A = P_{\mathbb{C}f_0}(L_{M_z}-M_z^*)|_{\clh}$ and $B = P_{\clh}(L_{M_z}-M_z^*)|_{\clh}$. Thus we only have to worry about the compactness of $B$. To this end, we consider the matrix representation of $B$ with respect to the orthonormal basis $\{f_n\}_{n \geq 1}$ as
\[
[B] = \begin{bmatrix}
{d_1} & (\frac{a_2}{a_1}-\frac{\bar{a}_1}{\bar{a}_2}) & -\bar{c}_1 & \frac{\bar{c}_1 {\bar{b}_3}}{\bar{a}_4}  & \ddots
\\
\frac{-d_1b_1}{a_2} & d_2 & (\frac{a_3}{a_2}-\frac{\bar{a}_2}{\bar{a}_3}) & -\bar{c}_2 & \ddots
\\
\frac{d_1b_1b_2}{a_2a_3} &\frac{-d_2b_2}{a_3} & d_3 & (\frac{a_4}{a_3}-\frac{\bar{a}_3}{\bar{a}_4}) & \ddots
\\
\frac{-d_1b_1b_2b_3}{a_2a_3a_4} &\frac{d_2b_2b_3}{a_3a_4} & \frac{-d_3b_3}{a_4} & d_4 & \ddots
\\
\vdots & \vdots & \vdots&\vdots &\ddots
\end{bmatrix}.
\]
In view of the above matrix representation, we define linear operators $B_1$, $B_2$ and $B_3$ on $\clh$, which admit the following matrix representations:
\[
[B_1] = \diag \Big( \frac{a_2}{a_1}-\frac{\bar{a}_1}{\bar{a}_2},\frac{a_3}{a_2}-\frac{\bar{a}_2}{\bar{a}_3}, \ldots \Big),
\]
and
\[
[B_2] = \begin{bmatrix}
-c_1 & 0 & 0 & 0 & \ddots
\\
\frac{c_1b_3}{a_4} & -c_2 & 0 & 0 & \ddots
\\
\frac{-c_1b_3b_4}{a_4a_5} & \frac{c_2b_4}{a_5} & -c_3 & 0 & \ddots
\\

\frac{c_1 b_3b_4b_5}{a_4a_5a_6} &\frac{-c_2 b_4b_5}{a_5a_6} & \frac{c_3b_5}{a_6} & -c_4 & \ddots
\\
\vdots & \vdots & \vdots&\ddots &\ddots
\end{bmatrix}
\text{ and }
[B_3] = \begin{bmatrix}
{d_1} & 0 & 0  & 0 & \ddots
\\
\frac{-d_1b_1}{a_2} & d_2 & 0 & 0 & \ddots
\\
\frac{d_1b_1b_2}{a_2a_3} &\frac{-d_2b_2}{a_3} & d_3 & 0 & \ddots
\\
\frac{-d_1b_1b_2b_3}{a_2a_3a_4} &\frac{d_2b_2b_3}{a_3a_4} & \frac{-d_3b_3}{a_4} & d_4 & \ddots
\\
\vdots & \vdots & \vdots&\vdots &\ddots
\end{bmatrix}.
\]
Assume for a moment that $B_1, B_2$ and $B_3$ are compact. Denote by $U$ the unilateral shift on $\clh$ corresponding to the orthonormal basis $\{f_n\}_{n \geq 1}$. In other words, $U f_n = f_{n+1}$ for all $n \geq 1$. Then
\[
B = B_1 U^* + B_2^* U^{*2} + B_3.
\]
Clearly, this would imply that $B$ is compact. Therefore, it suffices to show that $B_1, B_2$ and $B_3$ are compact operators. Note that there exist $\epsilon > 0$ and $M > 0$ such that
\begin{equation}\label{eqn: estimate of a/a}
\epsilon < \Big|\frac{a_n}{a_{n+1}}\Big| < M.
\end{equation}
Then
\[
\Big|\frac{a_{n+1}}{a_n} - \frac{\bar{a}_n}{\bar{a}_{n+1}} \Big| = \Big|\frac{a_{n+1}}{a_n}\Big(1- \Big|\frac{a_n}{a_{n+1}}\Big|^2\Big)\Big| < \frac{1}{\epsilon}\Big|1- \Big|\frac{a_n}{a_{n+1}}\Big|^2\Big|,
\]
implies that the sequence $\{|\frac{a_{n+1}}{a_n} - \frac{\bar{a}_n}{\bar{a}_{n+1}}|\}_{n \geq 0}$ converges to zero, which proves that $B_1$ is compact.

\NI We now prove that $B_2$ is compact. Since $\limsup |\frac{b_n}{a_{n+1}}|<1$, there exist $r \in (0,1)$ and $n_0 \in \N$ such that
\[
\Big|\frac{b_n}{a_{n+1}}\Big|< r \qquad (n \geq n_0).
\]	
Write
\[
\clh = (\bigoplus_{p=1}^{n_0-1} f_p) \oplus (\bigoplus_{q=0}^{\infty} f_{n_0+q}),
\]
and, with respect to this orthogonal decomposition, we let
\[
B_2 = \begin{bmatrix}
A_1 & 0
\\
A_3 & A_2
\end{bmatrix}.
\]
It is now enough to prove that $A_2$ acting on the infinite dimensional space $\oplus_{q=0}^{\infty} f_{n_0+q}$ is compact. Note
\[
[A_2] = \begin{bmatrix}
-c_{n_0} & 0 & 0 & \ddots
\\
\frac{c_{n_0} b_{n_0 + 2}}{a_{n_0 + 3}} & -c_{n_0 + 1} & 0 &  \ddots
\\
\frac{-c_{n_0} b_{n_0 + 2} b_{n_0 + 3}}{a_{n_0 +3}a_{n_0 +4}} &\frac{c_{n_0 + 1}b_{n_0 + 3}}{a_{n_0 + 4}} & -c_{n_0 + 2} &  \ddots
\\
\vdots & \vdots & \vdots &\ddots
\\
(-1)^n\frac{c_{n_0} b_{n_0 +2} \cdots b_{n_0 + n}}{a_{n_0 + 3}\cdots a_{n_0 + n + 1}} &(-1)^{n-1}\frac{c_{n_0 + 1}b_{n_0 + 3}\cdots b_{n_0 + n}}{a_{n_0 + 4}\cdots a_{n_0 + n + 1}} & (-1)^{n-2}\frac{c_{n_0 + 2}b_{n_0 + 4}\cdots b_{n_0 + n}}{a_{n_0 + 5}\cdots a_{n_0 + n + 1}} & \ddots
\\
\vdots & \vdots & \vdots &\ddots
\end{bmatrix}.
\]
Denote by $W_{n_0}$ the bounded weighted shift on $\oplus_{q=0}^{\infty} f_{n_0+q}$ with weight sequence $\{\frac{b_{n_0+n}}{a_{n_0 + n + 1}}\}_{n\geq 2}$, that is
\[
[W_{n_0}] = \begin{bmatrix}
0& 0 & 0 & 0 & \dots
\\
\frac{b_{n_0+2}}{a_{n_0+3}} & 0 & 0 & 0 & \ddots
\\
0 & \frac{b_{n_0+3}}{a_{n_0+4}} & 0 & 0 & \ddots
\\
0 & 0 & \frac{b_{n_0+4}}{a_{n_0+5}} & 0 &  \ddots
\\
\vdots & \vdots & \vdots & \ddots & \ddots
\end{bmatrix},
\]	
and write
\[
D_{n_0} = \diag (-c_{n_0} , -c_{n_0 + 1} , -c_{n_0 + 2}, \cdots).
\]
Suppose $M_0:=\sup_{n \geq 0} |c_n|$. Then
\[
\|D_{n_0}\| = \sup_{n \geq n_0}|c_n| \leq \sup_{n \geq 0}|c_n| = M_0,
\]
and, by the fact that $c_n \raro 0$, it follows that $D_{n_0}$ is a compact operator. Moreover, $A_2$ can be rewritten as
\[
A_2 = D_{n_0} - W_{n_0} D_{n_0} + W_{n_0}^2 D_{n_0}+ \cdots = \sum_{n=0}^{\infty} (-1)^n W_{n_0}^n D_{n_0}.
\]	
Clearly, $W_{n_0}^n D_{n_0}$ is compact for all $n \geq 0$, and, for $m \geq 2$, we have
\[
\begin{split}
\|W_{n_0}^m\| & \leq \sup_{l \geq 0}\Big|\frac{b_{n_0+2+l}b_{n_0+3+l}\cdots b_{n_0+m+l+1}}{a_{n_0+3+l}a_{n_0+4+l}\cdots a_{n_0+m+l+2}}\Big|
\\
& \leq r^m.
\end{split}
\]
Finally, consider the sequence $\{S_n\}_{n \geq 1}$ of partial sums of compact operators, where $S_n = \sum_{m=0}^{n} (-1)^m W_{n_0}^m D_{n_0}$ for all $n \geq 1$. Then
\[
\begin{split}
	\|A_2- S_n\| & = \|(-1)^{n+1}W_{n_0}^{n+1} D_{n_0} + (-1)^{n+2} W_{n_0}^{n+2} D_{n_0} + (-1)^{n+3} W_{n_0}^{n+3} D_{n_0} + \cdots\|
\\
& \leq M_0 \sum_{m=1}^{\infty}r^{n+m}
\\
& = M_1 r^n,
\end{split}
\]	
for some $M_1>0$ (as $0 < r  < 1$), and hence $A_2$ is the norm limit of a sequence of compact operators. This completes the proof of the fact that $B_2$ is compact.

\NI It remains to prove that $B_3$ is compact. First note that $d_{n+1}= - \frac{a_{n+2}}{a_n}c_n$ for all $n \geq 0$ (see \eqref{eqn: d =c}). The estimate \eqref{eqn: estimate of a/a} then implies that $c_n \raro 0$ if and only if $d_n \raro 0$. In particular, we may assume that $d_n \raro 0$. We are now in a similar situation as in the proof of the compactness of $B_2$. The proof of the fact that $B_3$ is compact now follows similarly as in the case of $B_2$.
\end{proof}	

\begin{Remark}
Note that if the sequence $\{\frac{b_n}{a_n}\}_{n\geq 0}$ is convergent, then $|\frac{b_n}{a_n} - \frac{b_{n+1}}{a_{n+1}}| \raro 0$. But the converse, evidently, is not true.
\end{Remark}

Note that if $b_n = 0$ for all $n \geq 0$, then $\clh_k$ is a diagonal space and $M_z$ on $\clh_k$ is a weighted shift. So in this case, Theorem \ref{trid M_z = S+K} recovers the classification of (the reproducing kernel version of) weighted shifts as obtained earlier in \eqref{eqn: S_w = C+I}. We refer the reader to \cite{AS} for the transition between weighted shifts and shifts on reproducing kernel Hilbert spaces.

\newsection{Concluding remarks}\label{sec: concl}

Let us now return to the general question (cf. Question \ref{Quest}) of quantitative classification of left-invertible shifts that are compact perturbations of isometries. Clearly, the equivalence in \eqref{eqn: S_w = C+I} and Theorem \ref{trid M_z = S+K} yields a complete answer to this question for the case of weighted shifts and tridiagonal shifts, respectively. In particular, if $M_z$ is the Bergman shift, or the weighted Bergman shift, or the Dirichlet shift, then \eqref{eqn: S_w = C+I} implies that $M_z=$ compact $+$ isometry.

However, unlike the diagonal case, it is not yet completely clear to us how to directly relate the kernel $k$ of the tridiagonal space $\clh_k$ to the conclusion of Theorem \ref{trid M_z = S+K}. In other words, our answer to Question \ref{Quest} for the tridiagonal case does not seem to indicate a comprehensive understanding (if any) of the general question.

To conclude this paper, we offer a general (but still abstract) classification of shifts that are compact perturbations of isometries. The proof is essentially a variant of Proposition \ref{thm: $T = S+K$}.

\begin{Proposition}\label{prop: factor}
Let $\clh_k$ be an analytic Hilbert space. Suppose the shift $M_z$ on $\clh_k$ is left-invertible and of finite index. Define $C$ on $\clh_k$ by
\[
(C f)(w) = \la f , (1-z\bar{w})k(\cdot, w) \ra_{\clh_k} \qquad (f \in \clh_k, w \in \D).
\]
Then $M_z =$ compact $+$ isometry if and only if $C$ defines a compact operator on $\clh_k$.
\end{Proposition}	
\begin{proof}
Since $M_z$ is left-invertible, the index of $M_z$ is negative. We know that  $M_z = \text{isometry}+ \text{compact}$ if and only if $I-M_zM_z^*$ is compact (Proposition \ref{thm: $T = S+K$}). A standard (and well known) computation shows that
\[
M_z^* k(\cdot, w) = \bar{w} k(\cdot, w) \qquad (w \in \D).
\]
Then
\[
(I -M_zM_z^*)k(\cdot, w) = (1 - z \bar{w}) k(\cdot, w) \qquad (w \in \D).
\]
For each $f \in \clh_k$ and $w \in \D$, we have $((I - M_zM_z^*)f)(w) = \la(I - M_zM_z^*)f , k(\cdot, w) \ra_{\clh_k} $, and hence
\[
((I - M_zM_z^*)f)(w) = \la f , (I - M_zM_z^*)k(\cdot, w)\ra_{\clh_k} = \la f , (1-z\bar{w}) k(\cdot, w)\ra_{\clh_k},
\]
which implies that $(I - M_zM_z^*)f = Cf$. This completes the proof.
\end{proof}	

On one hand, the above proposition is an effective tool for weighted shifts (the easy case, cf. \eqref{eqn: S_w = C+I}). For example, if $k$ is a diagonal kernel and
\[
k(z,w) = \frac{1}{1-z\bar{w}} \tilde{k}(z,w) \qquad (z,w \in \D),
\]
for some diagonal kernel $\tilde{k}$, then Proposition \ref{prop: factor} provides a definite criterion for answering Question \ref{Quest}. This is exactly the case with the Bergman and the weighted Bergman kernels. On the other hand, a quick inspection reveals that the (matrix) representation of $M_z M_z^*$ for a tridiagonal shift $M_z$ is rather complicated and the above proposition is less effective in drawing the conclusion as we did in Theorem \ref{trid M_z = S+K}.

Finally, it is worth pointing out that often Berezin symbols play an important role in proving compactness of linear operators on analytic Hilbert spaces \cite{Nordgren}. See \cite{Axler, Stroethoff, Suarez} and also \cite{Chalendar} for recent accounts on the theory Berezin symbols on analytic Hilbert spaces. However, in the present context, it is not clear what is the connection between Berezin symbols and compact perturbations of isometries.

\vspace{0.1in}

\noindent\textbf{Acknowledgement:}
The research of the second named author is supported in part by Core Research Grant, File No: CRG/2019/000908, by the Science and Engineering Research Board (SERB), Department of Science \& Technology (DST), Government of India.

\bibliographystyle{amsplain}

\end{document}